\documentclass[10pt,a4paper,twoside]{amsart}
\usepackage[T1]{fontenc}
\usepackage[utf8]{inputenc}
\usepackage[english]{babel}
\usepackage{amsmath}
\usepackage{amsthm}
\usepackage{amssymb}
\usepackage{amsfonts}
\usepackage{amsxtra}
\usepackage{enumerate}
\usepackage{verbatim}
\usepackage{color}
\usepackage[mathscr]{eucal}
\usepackage{graphicx}
\usepackage{xcolor}

\newcommand{\R}{\mathbb R}
\newcommand{\C}{\mathbb C}

\newcommand{\Sp}{\mathrm{Sp}(2n)}
\newcommand{\Se}{\mathrm{Sp}_{\mathrm{ell}}^+(2n)}

\newtheorem{theorem}{Theorem}[section]
\newtheorem{definition}[theorem]{Definition}
\newtheorem{lemma}[theorem]{Lemma}
\newtheorem{cor}[theorem]{Corollary}

\newtheorem*{remark}{Remark}
\theoremstyle{definition}

\let\phi=\varphi

\newcommand{\summe}[3]{\sum\limits_{#1}^{#2}#3}
\newcommand{\abb}[3]{#1\colon #2\rightarrow #3}
\newcommand{\real}[1]{\mathbb{R}^{#1}}
\newcommand{\rem}[1]{}

\DeclareFontFamily{U}{mathb}{\hyphenchar\font45}
\DeclareFontShape{U}{mathb}{m}{n}{
<-6> mathb5 <6-7> mathb6 <7-8> mathb7
<8-9> mathb8 <9-10> mathb9
<10-12> mathb10 <12-> mathb12
}{}
\DeclareSymbolFont{mathb}{U}{mathb}{m}{n}
\DeclareMathSymbol{\llcurly}{\mathrel}{mathb}{"CE}
\DeclareMathSymbol{\ggcurly}{\mathrel}{mathb}{"CF}

\title{A causal characterisation of $\Se$}
\author{Jakob Hedicke}
\address{Centre de recherches mathématiques (CRM), Université de Montréal,2920 Chemin de la tour,Montréal (Québec), H3T 1J4, Canada.}
\email{jakob.hedicke@gmail.com}
\date{\today}
\begin{document}

\begin{abstract}
We show that the natural conjugation invariant cone structure on the linear symplectic group $\Sp$ is globally hyperbolic in the positively elliptic region $\Se$.
This answers a question by Abbondandolo, Benedetti and Polterovich and shows a formula for a bi-invariant Lorentzian distance function defined by these authors for elements in this region.
Moreover we give a characterisation of the positively elliptic region and of $-id\in\Sp$ in terms of the causality of this cone structure.
\end{abstract} 
\maketitle
\section{Main results}

Let $\omega$ be the canonical symplectic form on $\real{2n}$, given in standard coordinates $(x_1,\cdots, x_n,y_1\cdots,y_n)$ by
$$\omega=\summe{i=1}{n}dx_i\wedge dy_i.$$
The \textit{linear symplectic group} $\Sp$ can be defined as the group of linear maps of $\real{2n}$ that preserve $\omega$.
The group $\Sp$ is a smooth Lie-group whose Lie-algebra is given by
$$\mathfrak{sp}(2n):=\{X\in \mathrm{Mat}(2n,\R)| \omega(\cdot,X\cdot) \text{ is symmetric}\}.$$

It is well known that the positive definite symmetric bilinear forms induce a convex cone in the Lie-algebra of $\Sp$ given by
$$\mathfrak{sp}^+(2n):=\{X\in \mathfrak{sp}(2n)|\omega(\cdot,X\cdot) \text{ is positive definite}\}.$$

Positive paths of symplectic maps, i.e. paths whose tangent vector lies in $\mathfrak{sp}^+(2n)$ play a key role in the study of linear Hamiltonian systems and their stability (see e.g. \cite{Krein50, Krein51, Krein55, Ekeland12, Lalonde97}).

The cone $\mathfrak{sp}^+(2n)$ induces a bi-invariant closed cone structure $C\subset T\Sp$ in the sense of \cite{Minguzzi192} by setting
$$C(id):=\{X\in \mathfrak{sp}(2n)|\omega(\cdot,X\cdot) \text{ is positive semi-definite}\}$$
and $C(W):=WC(id)=C(id)W$ for $W\in \Sp$ (see. e.g. \cite{Abbondandolo22},\cite{Lalonde97}).
This cone structure is the unique bi-invariant closed cone structure on $\Sp$ that is a proper subset of $T\Sp$.
In \cite{Abbondandolo22} the authors define a natural bi-invariant Lorentz-Finsler metric $G$ on $C$.
 
Note that while $(\mathrm{Sp}(2),G)$ is isometric to the $3$-dimensional anti-de Sitter space $\mathrm{AdS}_3$, for $n>1$ the Lorentz-Finsler metric $G$ is not induced by a Lorentzian metric.
In particular $\partial C$ is not smooth and $G$ is smooth in the interior of $C$, but only extends continuously to $0$ at $\partial C$.

Moreover in \cite{Abbondandolo22} it is shown, that the timelike geodesics of $G$ are of the form $e^{tX}W$, where $W\in \Sp$ and $X\in \mathfrak{sp}^+(2n)$. 
It follows that $(\Sp,G)$ is totally vicious (i.e. there is a closed timelike curve through every point) since given $W\in\Sp$, the curve $e^{tJ}W$ is timelike and closed.
Here $J$ denotes the standard $\omega$-compatible complex structure on $\R^{2n}$.

Of particular interest for the theory of linear Hamiltonian systems are strongly stable systems.
These are linear Hamiltonian systems whose solution remains bounded for all times such that this property survives small perturbations of the system (see \cite{Ekeland12}).

Such systems can be studied in terms of Krein-theory \cite{Krein50,Krein51,Krein55}.
In particular a linear symplectic map gives rise to a strongly stable Hamiltonian system if all its (complex) eigenvalues lie on $S^1\setminus\{\pm 1\}$ and the Krein form $\kappa(v,w):= \langle iJ v,w\rangle$ is positive definite on the (complex) eigenspaces of eigenvalues with positive imaginary part.
Here $\langle\cdot,\cdot\rangle$ denotes the standard Hermitian product on $\C^{2n}$ and $J$ the canonical complex structure on $\R^{2n}$.

The set of all such symplectic maps is an open subset of $\Sp$ called the \emph{positively elliptic region} $\Se$.

Equivalently $W\in \Se$ if and only if there exists a splitting of $\R^{2n}$ into symplectic planes $V_1,\cdots,V_n$ and $\omega$-compatible complex structures $J_k$ on $V_k$ such that 
$$W =\bigoplus\limits_{k=1}^n e^{\theta_kJ_k},$$
where $\theta_k\in (0,\pi)$ (see \cite[Proposition ii.2.]{Abbondandolo22}).

The most important causal condition in the study of cone structures is global hyperbolicity.
A closed cone structure $(M,C)$ is globally hyperbolic if and only if there exists a Cauchy surface, i.e. an embedded hypersurface $\Sigma\subset M$ such that any inextensible curve tangent to the cone intersects $\Sigma$ in a unique point.
In the following we show

\begin{theorem}\label{thm1}
The positively elliptic region $\Se$ is globally hyperbolic. 
\end{theorem}

This answers \cite[Question J.5.]{Abbondandolo22}.
Moreover as a consequence we can generalise \cite[Theorem J.1 (i)]{Abbondandolo22}.
Recall that in \cite{Abbondandolo22} the authors introduced a bi-invariant Lorentz-Finsler metric on the interior of $C$.
It is defined on $\mathfrak{sp}^+(2n)$ as
$$G(X):=\det(X)^{\frac{1}{2n}}.$$
This allows to define the Lorentzian-length of a curve tangent to $\mathrm{Int}(C)$.

Consider the universal cover $\widetilde{\Sp}$.
Let $\widetilde{\Se}$ be the unique lift of $\Se$ whose boundary contains the identity in $\widetilde{\Sp}$. 
The cone structure on $\Sp$ lifts to a causal cone structure on $\widetilde{\Sp}$ (see \cite[Theorem I.1.]{Abbondandolo22}), i.e. a cone structure without any closed curves tangent to the cone.
The Lorentz-Finsler metric lifts to $\widetilde{\Sp}$ and induces a Lorentzian distance function $\mathrm{dist}_G$ defined for $w_1,w_2\in\widetilde{\Sp}$ as the supremum over the lengths of all curves tangent to $C$ connecting $w_1$ and $w_2$ (see e.g. \cite{Beem} for a detailed treatment of Lorentzian distances and their properties).

\begin{cor}\label{cor1}
Let $w$ be in the closure of $\widetilde{\Se}$ with eigenvalues 

$\{e^{\pm i\theta_1},\cdots, e^{\pm i\theta_n}\}$ for $\theta_k\in [0,\pi]$.
Then 
$$\mathrm{dist}_G(id,w)=\sqrt[n]{\theta_1\cdots\theta_n}.$$
\end{cor}

Similar arguments provide a causal characterisation of $\Se$ and $-id\in\Sp$.
Denote by $I^+(w)$ and $I^-(w)$ the chronological future and past of $w\in\widetilde{\Sp}$, i.e. the set of points that are connected to $w$ by a curve tangent to $\mathrm{Int}(C)$ and $-\mathrm{Int}(C)$, respectively.
\newpage
\begin{theorem}\label{thm2}
Let $\widetilde{-id}$ be the unique lift of $-id$ lying on the boundary of $\widetilde{\Se}$.
Then
\begin{itemize}
\item[i)] $\widetilde{\Se}=I^+(id)\cap I^-(\widetilde{-id})$.
\item[ii)] If $I^+(id)\cap I^-(w)\neq\emptyset$ is globally hyperbolic for some $w\in \widetilde{\Sp}$ then $w$ lies in the closure of $\widetilde{\Se}$.
\item[iii)] If $I^+(id)\cap I^-(w)=\widetilde{\Se}$ then $w=\widetilde{-id}$.
\end{itemize}
\end{theorem}

\begin{remark}
Theorem \ref{thm2} i) and ii) in particular show that $I^+(id)\cap I^-(w)\subset \widetilde{\Se}$, i.e. $\widetilde{\Se}$ is the union of all sets of the form $I^+(id)\cap I^-(w)$, where $w\in \widetilde{\Sp}$ is such that $I^+(id)\cap I^-(w)$ is globally hyperbolic.
\end{remark}

\textbf{Acknowledgements.}

I would like to thank Alberto Abbondandolo and Stefan Suhr for introducing me to this question and for many fruitful discussions.
Moreover I would like to thank Ettore Minguzzi for suggesting a nice time function on $\Se$.
This research was supported by the SFB/TRR 191 ``Symplectic Structures in Geometry, Algebra and Dynamics'', funded by the Deutsche Forschungsgemeinschaft (Projektnummer 281071066 – TRR 191) and the Centre de recherches mathématiques (CRM).

\section{Causality theory for closed cone structures}

In this section we recall some basic facts about causality theory and cone structures.
Closed cone structures are a natural generalisation of Lorentzian manifolds.

\begin{definition}
A closed proper cone $C\subset V$ in a vector space $V$ is a closed convex subset of $V\setminus \{0\}$ with non-empty interior such that $\lambda v\in C$ if $v\in C$ and $\lambda>0$.
A closed proper cone structure on a smooth manifold $M$ is a closed subset $C\subset TM$ such that $C(p):=C\cap T_pM$ is a closed proper cone for all $p\in M$.
\end{definition}

\begin{remark}
The set
$$C:=\bigcup\limits_{W\in\Sp}C(W)$$
defines a closed proper cone structure on $\Sp$.
This cone structure is Lipschitz in the sense of \cite{Fathi12, Minguzzi192}. 
\end{remark}

Similar to Lorentzian geometry, a cone structure on a smooth manifold $M$ gives rise to the notion of causality.
A smooth curve $\gamma$ is called \textit{timelike} if $\gamma'\subset \mathrm{Int}(C)$, \textit{null} if $\gamma'\in \partial C\setminus \{0\}$ and \textit{causal} if $\gamma'\subset C$.
Then the \textit{chronological future and past} of a point $p$ can be defined as
\begin{align*}
I^+(p):=\{q\in M|\text{ there exists a timelike curve from $p$ to $q$}\}\\
I^-(p):=\{q\in M|\text{ there exists a timelike curve from $q$ to $p$}\}
\end{align*}

and the \textit{causal future and past} as

\begin{align*}
J^+(p):=\{q\in M|\text{ there exists a causal curve from $p$ to $q$}\}\cup \{p\}\\
J^-(p):=\{q\in M|\text{ there exists a causal curve from $q$ to $p$}\}\cup \{p\}.
\end{align*}

Finally the \textit{future/past horismos} of $p$ is defined as $E^{\pm}(p):=J^{\pm}(p)\setminus I^{\pm}(p)$.
Given a set $S\subset M$ we define $I^{\pm}(S)$ and $J^{\pm}(S)$ to be the union of all chronological (causal) futures/pasts of points in $S$.

The cone structure $(M,C)$ is called \textit{causal} if there are no closed causal loops and \textit{strongly causal} if every point $p\in M$ admits an open neighbourhood $U$ such that any causal curve with endpoints in $U$ is contained in $U$ ($U$ is causally convex).
Obviously strongly causal cone structures are causal.
If there exists a time function on $(M,C)$, i.e. a continuous function strictly increasing along causal curves then we call $(M,C)$ \textit{stably causal}.
Note that stably causal cone structures are strongly causal.

The strongest causal condition for cone structures is global hyperbolicity.
A cone structure $(M,C)$ is called \textit{globally hyperbolic} if it is causal and the sets $J^{+}(p)\cap J^-(q)$ are compact for all $p,q\in M$.
Global hyperbolicity in particular implies the existence of a Cauchy surface $\Sigma$, i.e. a hypersurface in $M$ that is intersected in a unique point by any inextensible causal curve and a smooth splitting $M\cong \R\times \Sigma$ (see e.g. \cite{Bernard18, Minguzzi192}).

Given a cone structure one can define the notion of null geodesic.
In the case of Lorentzian metrics this notion coincides (up to parametrisation) with the usual one.

\begin{definition}
A null geodesic of a cone structure $C$ is a continuous curve $\abb{\gamma}{I}{N}$ that is locally horismotic.
This means that for every $t_0\in I$ and any open neighbourhood $U$ of $\gamma(t_0)$ there exists an open neighbourhood $\gamma(t_0)\in V\subset U$ such that if $\gamma([t_0-\epsilon,t_0+\epsilon]\cap I)\subset V$, then $\gamma(s_1)\in E_{C_U}^+(\gamma(s_0))$ for all $s_0,s_1\in [t_0-\epsilon,t_0+\epsilon]\cap I$ with $s_0<s_1$.
Here $C_U$ denotes the restriction of the cone structure $C$ to the neighbourhood $U$.
\end{definition}

In the case of strongly causal cone structures a null geodesic is just a curve $\gamma(t)$ with the property that for any $t_0$ there exists an $\epsilon>0$ such that $\gamma(t_2)\in E^+(\gamma(t_1))$ for all $t_0-\epsilon<t_1\leq t_2<t_0+\epsilon$. 

\section{The positively elliptic region and the Maslov quasimorphism}
In this section we introduce some basic facts about $\Se$ and the Maslov quasimorphism $\abb{\mu}{\widetilde{\Sp}}{\R}$ defined on the universal cover of the symplectic group.

Recall that by $\Sp$ we denote the linear symplectic group, i.e. the group of linear maps preserving the canonical symplectic form on $\R^{2n}$.
As mentioned in the introduction the positively elliptic region $\Se$ can be defined as an open subset of strongly stable symplectic maps.
Another characterisation can be given in terms of Krein theory (see e.g. \cite{Ekeland12}):

The canonical symplectic form on $\R^{2n}$ can be extended to a skew-hermitian form 
$$\omega\colon\C^{2n}\times \C^{2n}\rightarrow\C^{2n}.$$
The \textit{ Krein form} can be defined as
$$\kappa:=-i\omega.$$
Note that $\kappa$ is a Hermitian form on $\C^{2n}$.

Consider $W\in\Sp$ as a linear map on $\C^{2n}$.
Given a complex eigenvalue $\lambda$, the form $\kappa$ is non-degenerate on the generalised eigenspace $E_{\lambda}$ if and only if $\lambda$ lies on the unit circle.
Another eigenvalue of $W$ is $\overline{\lambda}$.
If $\lambda$ lies on the unit circle and $\kappa$ has signature $(p,q)$ on $E_{\lambda}$, it has signature $(q,p)$ on $E_{\overline{\lambda}}$.
In particular the signature of $\kappa$ on $E_{\pm1}$ is always $(p,p)$.

\begin{definition}
The positively elliptic region $\Se$ is defined as the set of $W\in\Sp$ such that all (complex) eigenvalues of $W$ are contained in $S^1\setminus\{\pm 1\}$ and $\kappa$ is positive definite on $E_{\lambda}$ if and only if $\lambda$ is an eigenvalue of $W$ with positive imaginary part.
\end{definition}

Due to the definiteness of $\kappa$ on the eigenspaces, the positively elliptic region forms an open subset of $\Sp$.
Moreover all maps in $\Se$ are diagonalisable as \cite[Proposition ii.2.]{Abbondandolo22} shows.
The proposition states that the exponential map $\mathrm{exp}\colon \mathfrak{sp}(2n)\rightarrow \Sp$ maps the set
$$\mathfrak{sp}(2n)_{\mathrm{ell}}^+:=\{X\in\mathfrak{sp}^+(2n)|\sigma(X)\subset i(-\pi,\pi)\}$$
diffeomorphically to $\Se$.
Moreover $W\in\Se$ if and only if there exists a  $W$-invariant splitting of $\R^{2n}$ into symplectic planes $V_k$ such that 
$$W =\bigoplus\limits_{k=1}^n e^{\theta_kJ_k}.$$
Here $\theta_k\in (0,\pi)$ and $J_k$ denotes an $\omega$-invariant complex structure on $V_k$.

Due to the non-degeneracy of the Krein form causal paths behave particularly nice on $\Se$ (\cite{Ekeland12}).
For instance it is known that along a timelike curve the eigenvalues with positive imaginary part rotate counterclockwise, while the once with negative imaginary part rotate clockwise on $S^1$.
This means that the $\theta_k$ in the splitting above are increasing along causal curves.
Moreover any causal path $W(t)$ can leave $\Se$ only once the eigenvalue $-1$ is attained. 

The behaviour of eigenvalues along a causal curve can be estimated using the Maslov quasimorphism which we briefly introduce.
For details see e.g. \cite{Abbondandolo01, Abbondandolo22}.

Denote by $\widetilde{\Sp}$ the universal cover of the symplectic group.
The universal cover has a natural group structure with neutral element being a lift of the identity.
We denote the neutral element by $id$ and by $\widetilde{\Se}$ the unique lift of $\Se$ such that $id$ lies on the boundary.

The Maslov quasimorphism 
$$\mu\colon \widetilde{\Sp}\rightarrow \R$$
was introduced in \cite{Gelfand58} as the unique homogeneous real quasimorphism whose restriction to the lift of $\mathrm{U}(n)\subset\Sp$ agrees with the complex determinant.
Here the term quasimorphism means that there exists a constant $C$ such that 
$$|\mu(vw)-\mu(v)-\mu(w)|\leq C$$
for all $v,w\in\widetilde{\Sp}$.
In \cite{Abbondandolo22} this quasimorphism is used to construct a time function on $\widetilde{\Sp}$.

For $W\in \Sp$ define
$$\nu(W):=(-1)^m\prod\limits_{\lambda\in \sigma(W)\cap S^1\setminus\{\pm 1\}}\lambda^{p(\lambda)},$$
where $2m$ is the total algebraic multiplicity of all negative real eigenvalues and $(p(\lambda),q(\lambda))$ the signature of $\kappa$ on $E_{\lambda}$.
As pointed out in \cite{Abbondandolo22} the Maslov quasimorphism can be computed as the unique continuous function satisfying
$$\mu(w)=e^{2\pi i\nu(W)}, \mu(id)=0.$$
Here $W$ denotes the projection of $w$ to $\Sp$.

Given $W\in\Se$ denote by $\{\theta_1^W,\cdots, \theta_n^W\}$ the arguments of the eigenvalues of $W$ with positive imaginary part.
Let $w\in\widetilde{\Se}$ be the lift of $W$.
Then 
$$\mu(w)= \frac{1}{2\pi}(\theta^W_1+\cdots+\theta^W_n).$$

\section{Proofs}

As before for $W\in\Se$ denote by $\{\theta_1^W,\cdots, \theta_n^W\}$ the arguments of the eigenvalues of $W$ with positive imaginary part.
Define 
\begin{align*}
&\abb{\tau}{\Se}{\R}\\
&W\mapsto \sum\limits_{i}\ln (\theta^W_i)-\ln(\pi-\theta^W_i)
\end{align*}

\begin{lemma}\label{lem1}
The function $\tau$ is a continuous time function on $\Se$, i.e. $\tau$ is strictly increasing along causal curves.
\end{lemma}

\begin{remark}
In fact the proof of Theorem \ref{thm1} implies that $\tau$ is a Cauchy time function, i.e. its level sets are (topological) Cauchy surfaces.
\end{remark}

\begin{proof}
Let $W(t)$ be a causal curve in $\Se$.
It follows from \cite[Proposition 2]{Ekeland12} that $\tau$ is non-decreasing along $W(t)$.
Since $W'(t)\in C(W(t))$ there exists an eigenvector $x$ for some eigenvalue $e^{i\theta}$ of $W(t)$ with $\omega(x,W'(t)x)>0$ and $\kappa(x,x)=1$.
Following the proof of \cite[Proposition 2]{Ekeland12} take a sequence $t_k\searrow t$.
Then there exists a sequence of eigenvalues $e^{i\theta_k}$ for $W(t_k)$ 
with $\theta_k\rightarrow \theta$ and a sequence of eigenvectors $x_k$ with $\kappa(x_k,x_k)=1$ and $x_k\rightarrow x$.

As pointed out in \cite{Ekeland12} it follows that
$$\omega\left(\frac{W(t)-W(t_k)}{t-t_k} x, x_k\right)=\frac{(e^{i\theta}-e^{i\theta_k})}{t-t_k}\omega(x,x_k).$$
Taking the limit this implies that for $t_k$ close to $t$ we have that $\theta_k>\theta$, i.e. $ \tau(W(t_k))>\tau(t)$.
\end{proof}

\begin{lemma}\label{lem3}
Let $X\in C$ and $W_0\in \Se$. 
Then there exist $0<c_1,c_2<\infty$ such that $e^{tX}W_0\in \Se$ for all $t\in (-c_1,c_2)$ and $e^{-c_1X}W_0,e^{c_2X}W_0\notin \Se$.
In particular 
$$\lim\limits_{t\searrow -c_1}\tau(e^{tX}W_0)=-\infty, \lim\limits_{t\nearrow c_2}\tau(e^{tX}W_0)=\infty$$
\end{lemma} 

\begin{proof}
Let $\theta_1(t),\cdots, \theta_n(t)$ be the arguments of the Krein positive eigenvalues of $e^{tX}W_0$.

1. case:
$X$ has an eigenvalue $\lambda\neq 0$.

Let $w(t)\subset \widetilde{\Sp}$ be a lift of $e^{tX}$ starting at $id$.
Then the Maslov quasimorphism satisfies $\mu(w(t))=tc$ for some $c>0$, i.e. $\mu(w)$ diverges.
Let $w_0$ be a lift of $W_0$. 
Since $\mu$ is a quasimorphism there exists a constant $C\geq 0$ such that $\mu(w(t)w_0)\in (\mu(w_0)+tc-C, \mu(w_0)+tc+C)$.
Hence $\mu(w(t)w_0)$ diverges.
It follows that $e^{tX}W_0$ takes the eigenvalues $1$ and $-1$ for infinitely many $t$, in particular it leaves $\Se$ in the future and in the past for finite $t$.

2. case:
The only eigenvalue of $X$ is $0$.

In this case $X$ is nilpotent since one has $(X)^{2n}=0$.
It follows that
$$e^{tX}W=W+tXW+\cdots + \frac{t^{2n-1}}{(2n-1)!}(X)^{2n-1}W.$$
In particular the characteristic polynomial $P_t$ is a monic polynomial of degree $2n$ whose coefficients are polynomials in $t$.
Write
$$P_t(\lambda)=\lambda^{2n}+P_{2n-1}(t)\lambda^{2n-1}+\cdots + P_{1}(t)\lambda+1.$$
Then due to \ref{lem1} at least one of the eigenvalues of $e^{tX}W$ moves counter-clockwise on the unit circle, i.e. there exists some non-constant $P_k(t)$.

Over $\C$ the polynomial $P_t(\lambda)$ can be factorized, i.e. the coefficients $P_i(t)$ depend polinomially on the eigenvalues of $e^{tX}W$.
Since $\lim\limits_{t\rightarrow \pm\infty} |P_k(t)|=\infty$, it follows that at least one of the eigenvalues of $e^{tX}W$ diverges for $t\rightarrow \pm\infty$.
In particular $e^{tX}W$ leaves $\Se$ for finite $t$ in the future and in the past.
\end{proof}

\begin{lemma}\label{lem2}
The positively elliptic region $\Se$ is causally geodesically connected, i.e. for $W_0\in \Se$ and $W_1\in J^+(W_0)$ there exists $X\in C$ with $W_1=e^{X}W_0$ and $e^{sX}W_0\in \Se$ for $s\in [0,1]$.
\end{lemma}

\begin{proof}
Claim: For $W_0\in \Se$ and $W_1\in I^+(W_0)$ we have that $W_1W_0^{-1}\in \Se$.

Suppose not.
Let $W(t)$ be a timelike curve in $\Se$ with $W(0)=W_0$ and $W(1)=W_1$.
Due to \cite[Proposition ii.2.]{Abbondandolo22} there exists $X_0\in C$ with $W_0=e^{X_0}$ and $e^{sX_0}\in \Se$ for all $s\in (0,1]$.
Define
\begin{align*}
H\colon [0,1]\times [0,1]\rightarrow \Sp\\
(s,t)\mapsto W(t)e^{-sX_0}.
\end{align*}
Then $\frac{d}{dt}H(s,t)\in C(H(s,t))$ and $\frac{d}{s}H(s,t)\in -C(H(s,t))$.
Note that $H(s,t)\in \Se\cup\{id\}$ for all $s$ and $t$ sufficiently small.
Since $\Se$ is open, the set of $s$ such that the curve $t\mapsto H(s,t)$ is contained in $\Se$ is open and non-empty.
Let
$$s_0:=\min \{s\in [0,1]| t\mapsto H(s,t) \nsubseteq  \Se \}.$$

Then $t\mapsto H(s_0,t)$ leaves $\Se$ at some $t_0$.
By \cite[Proposition 1]{Ekeland12} $H(s_0,t)\in \Se$ for all $t\neq 0$ small.
Assume that $t_0>0$ is minimal. 
Since $t\mapsto H(s_0,t)$ is contained in $\Se$ for $t\in (0,t_0)$ and is timelike, it follows from \cite[Proposition 2]{Ekeland12} that $H(s_0,t_0)$ does not have eigenvalue $1$.
On the other hand the curve $s\mapsto H(s_0-s,t_0)$ is causal and contained in $\Se$ for $s\in [0,s_0)$.
Hence by \cite[Proposition 2]{Ekeland12} $H(s_0,t_0)$ must have the eigenvalue $1$ which contradicts the previous observation.

Take $X_1\in C$ as in  \cite[Proposition ii.2.]{Abbondandolo22} with $e^{X_1}=W_1W_0^{-1}$ and $e^{sX_1}\in \Se$ for all $s\in (0,1]$.
Then $e^{X_1}W_0$ satisfies the assumptions of the Lemma.
Since $e^{X_1}\in I^+(e^{tX_1})$ for all $t\in [0,1)$ it follows analogously to the previous claim that $e^{sX_1}W_0=e^{(s-1)X_1}e^{X_1}W_0\in \Se$ for all $s\in (0,1]$.

Now let $W_1\in J^+(W_0)$.
It follows from the fact that $(\Se,C)$ is locally Lipschitz and \cite[Theorem 2.7]{Minguzzi192} that $J^+(W_0)\subset \overline{I^+(W_0)}$.
Choose a Riemannian metric on $\Sp$ and let $D\subset C(id)$ be the intersection of $C(id)$ with the unit sphere.
Then there exists $X_n\in D$ and $t_n$ with $e^{t_nX_n}W_0\rightarrow W_1$.
A subsequence of $X_n$ converges to some $X\in D$.

Suppose there is no $t<\infty$ such that $e^{tX}W_0=W_1$ and $e^{sX}W_0\in \Se$ for all $0<s<t$.
By \ref{lem3} the curve $e^{tX}W_0$ leaves $\Se$ for finite $t$, i.e. for all $c>0$ there exists some $t_{c}$ with $\tau(e^{t_{c}X}W_0)=c$.
In particular for $c>\tau(W_1)$ one has $\tau(e^{t_cX}W_0)>\tau(W_1)$.
On the other hand it follows from the limit curve theorem  (\cite[Theorem 2.14]{Minguzzi192}) that $e^{t_cX}W_0\in \overline{J^-}(W_1)$.
This contradicts $\tau$ being a time function.

Thus $W_1=e^{t_1X}W_0$ for some $t_1<\infty$.
\end{proof}

\begin{proof}[Proof of Theorem \ref{thm1}]
It remains to show that $J^+(W_0)\cap J^-(W_1)$ is compact for all $W_0,W_1\in\Se$.

Let $A_k$ be a sequence in $J^+(W_0)\cap J^-(W_1)$.
Like before define $D$ to be the intersection of $C(id)$ with a unit sphere.
By Lemma \ref{lem2} we can choose sequences $X_k,Y_k\in D$ and $t_k,s_k>0$ with $A_k=e^{t_kX_k}W_0=e^{-s_kY_k}W_1$.
Due to compactness $X_k\rightarrow X\in D$ and $Y_k\rightarrow Y\in D$ up to a subsequence.
Define $c_1(X),c_2(X)$ and $c_1(Y),c_2(Y)$ as in Lemma \ref{lem3}

For $t_k$ there are two cases:
\begin{itemize}
\item[a)] $$\lim\limits_{k\rightarrow \infty}t_k=t< c_2(X).$$
In this case $A_k\rightarrow e^{tX}\in J^+(W_0)$.
\item[b)] $$\liminf\limits_{k\rightarrow \infty}t_k\geq c_2(X).$$
In this case $\lim\limits_{k\rightarrow \infty}\tau(A_k)\rightarrow \infty$.
\end{itemize}

Similarly for $s_k$ one has
\begin{itemize}
\item[a)'] $$\lim\limits_{k\rightarrow \infty}s_k=s< c_1(Y).$$

In this case $A_k\rightarrow e^{-sY}\in J^-(W_1)$.
\item[b)'] $$\liminf\limits_{k\rightarrow \infty}t_k\geq c_1(X).$$
 In this case $\lim\limits_{k\rightarrow \infty}\tau(A_k)\rightarrow -\infty$.
\end{itemize}

One can easily see that if b) holds than neither a)' nor b)' can hold.
Similarly if b)' holds than neither a) nor b) can hold.
It follows that a) and b) hold, i.e. a subsequence of $A_k$ converges to some $A\in J^+(W_0)\cap J^-(W_1)$.

\end{proof}

All points and subsets in the following proofs are contained in the universal cover of $\Sp$.
To simplify notation we will from now on omit all tildes and e.g. denote $\widetilde{\Se}$ by $\Se$.

\begin{proof}[Proof of Corollary \ref{cor1}]
Let $W\in\Se$ with eigenvalues 

$\{e^{\pm i\theta_1},\cdots, e^{\pm i\theta_n}\}$ for $\theta_k\in (0,\pi)$.
Due to the bi-invariance of the cone structure we can pick some $A\in I^-(id)$ sufficiently close to $id$ such that $A\Se$ is globally hyperbolic and $id,W\in A\Se$.
Since $A\Se$ is causally convex and globally hyperbolic the Avez-Seifert Theorem implies that there exists a timelike geodesic between $id$ and $W$ maximising the Lorentzian length in $\Sp$ (the proof works as in the Lorentzian case \cite{Hawking}, see e.g. \cite{Minguzzi192} for the case of cone structures), i.e. a curve $W(t)$ with $W(0)=id$, $W(1)=W$ and 
$$\mathrm{dist}_G(id,W)=\int\limits_0^1G(W'(t))dt.$$
As shown in \cite{Abbondandolo22} all timelike geodesics are of the form $e^{tX}B$ for some $B\in\Sp$ and $X\in \mathfrak{sp}^+(2n)$.
Moreover \cite[Proposition ii.2.]{Abbondandolo22} implies that there exists a unique $X\in\mathfrak{sp}^+(2n)$ with $W=e^X$.
There exists a splitting of $\R^{2n}$ into symplectic planes $V_k$ such that the map $X$ is of the form 
$$X=\bigoplus\limits_{k=1}^n \theta_kJ_k$$
for some $\omega$-compatible complex structures $J_k$ on $V_k$.
It follows that $G(X)=\sqrt[n]{\theta_1\cdots\theta_n}$ is the length of the timelike geodesic $t\mapsto e^{tX}$ between $id$ and $W$.

The causal convexity of $\Se$ in the universal cover implies that this the unique timelike geodesic between $id$ and $W$ and therefore maximising, i.e. 
$$\mathrm{dist}_G(id,W)=\sqrt[n]{\theta_1\cdots\theta_n}.$$

For $W\in \overline{\Se}$ the result follows from the lower semi-continuity of $\mathrm{dist}_G$.
\end{proof}

\begin{proof}[Proof of Theorem \ref{thm2}]

\begin{itemize}

\item[i)]
It follows from \cite[Proposition ii.2.]{Abbondandolo22} and the analogous statement for $-id$ that $\Se\subset I^+(id)\cap I^-(-id)$.

Let $W\in  I^+(id)\cap I^-(-id)$.
Assume $W\notin \Se$.
Let $\gamma_1(t)$ be a timelike curve from $id$ to $W$ and $\gamma_2(t)$ be a timelike curve from $W$ to $-id$.
By \cite[Proposition ii.2.]{Abbondandolo22} and the analogous statement for $-id$, $\gamma_1$ and $\gamma_2$ have to enter $\Se$, i.e. there exist $t_1,t_2$ with $\gamma_1(t_1),\gamma_2(t_2)\in \partial\Se$.
Assume that $t_1$ and $t_2$ are minimal, respectively maximal with this property.
Then by \cite[Proposition 2]{Ekeland12} $\gamma_1(t_1)$ is the lift of a map that has eigenvalue $-1$ but not eigenvalue $1$.
Similarly $\gamma_2(t_2)$ is the lift of a map that has eigenvalue $1$ but not eigenvalue $-1$.
Considering the time function $f$ on $\Sp$ constructed in \cite{Abbondandolo22} this implies $f(\gamma_1(t_1))>f(\gamma_2(t_2))$.
This contradicts $\gamma_2(t_2)\in I^+(\gamma_1(t_1))$.

\item[ii)]

Clearly $W\in I^+(id)$.
Suppose $W\notin \overline{I^-(-id))}$.
One can choose $V\in I^-(id)$ such that $VW\notin \overline{I^-(-id))}$.
Due to the conjugation invariance of the cone structure $I^+(Id)\cap I^-(W)$ is globally hyperbolic if and only if $I^+(V)\cap I^-(WV)$ is globally hyperbolic.
Note that $id \in I^+(V)\cap I^-(WV)$.
Moreover there exist $W_1\in I^+(id)\cap I^-(WV)\cap\partial \Se$ with $J^+(id)\cap J^-(W_1)\neq\emptyset$ being compact.

Claim: There exists an inextensible causal curve $\gamma(t)$ through $W_1$ that does not have eigenvalue $1$ for all $t$.

First note that since $W_1\in I^+(id)\cap\partial \Se$ and $J^+(id)\cap J^-(W_1)$ is compact, there is an $X\in C(id)$ with $W_1=e^X$:
Choose a timelike curve $a(t)$ with $a(0)=id$ and $a(1)=W_1$ and a sequence $t_n\subset (0,1)$ with $t_n\rightarrow 1$
Define $D\subset C(id)$ like in the proof of Lemma \ref{lem2} as the intersection of the cone with a unit sphere.
Lemma \ref{lem2} implies the existence of $X_n\in D$ and $s_n$ with $e^{s_n X_n}=a(t_n)$.
Since $D$ is compact, a subsequence of $X_n$ converges to some $X\in D$.
Moreover due to compactness and the stable causality of the universal cover, the curve $s\mapsto e^{sX}$ leaves $J^+(id)\cap J^-(W_1)$ at some $s_{\infty}<\infty$.
It follows that $s_n\rightarrow s_{\infty}$ and $e^{s_{\infty}X}=W_1$.
By eventually perturbing $W_1$ we can assume that $X$ is timelike.

Proposition ii.2. in \cite{Abbondandolo22} implies that there exists a $W_1$-invariant splitting of $\R^{2n}$ into symplectic planes $V_k$ such that 
$$W_1 =\bigoplus\limits_{k=1}^n e^{\theta_kJ_k}.$$
Here $J_k$ denotes an $\omega$-compatible complex structure on $V_k$ and $\theta_k\in (0,\pi]$. 
Note that there exists a $j$ such that $\theta_j=\pi$.
W.l.o.g assume $j=n$.

For $n=1$ any point in $I^+(id)\cap \partial \Se$ lies on a null geodesic consisting of maps that only have eigenvalue $-1$(\cite{Abbondandolo01, Abbondandolo22}).

Look at the the causal curve
$$\gamma(t):=\bigoplus\limits_{k=1}^{n-1} e^{\theta_kJ_k}\oplus b(t),$$
where $b(t)$ denotes a null geodesic in $V_n$ with $b(0)= e^{\theta_nJ_n}$ that has only eigenvalue $-1$.
Then $\gamma(t)$ is null, non-constant and does not have Eigenvalue $1$ for all $t$.
Moreover $\gamma(t)\in\partial\Se$ for all $t\leq 0$.

The compactness of $J^+(id)\cap J^-(W_1)$ implies that there exists $t_0<0$ such that $\gamma(t_0)\in \partial J^+(id)$.
Like before there exists some $Y\in C(id)$ such that $\gamma(t_0)=e^Y$.
Since $\gamma(t_0)\in \partial J^+(id)$, one has $Y\in \partial C(id)$.
This contradicts $\gamma(t_0)$ not having eigenvalue $1$.

\item[iii)] 

Assume that $I^+(id)\cap I^-(w)=\Se$.
Then $w\in \overline{I^-(-id)}$ and $-id\in \overline{I^-(w)}$.
The existence of a time function on $\Sp$ implies that $w=-id$.

\end{itemize}
\end{proof}

\end{document}